\documentclass[11pt,reqno]{amsart}

\usepackage[utf8]{inputenc} 

\usepackage[margin=1in]{geometry} 

\usepackage{graphicx} 
\usepackage{float} 

\usepackage[parfill]{parskip} 

\usepackage{booktabs} 
\usepackage{array} 
\usepackage{paralist} 
\usepackage{verbatim} 
\usepackage{subfig} 
\usepackage{mathrsfs}
\usepackage{amssymb}
\usepackage{amsthm}
\usepackage{amsmath,amsfonts,amssymb,esint,hyperref}
\usepackage[noabbrev, capitalize]{cleveref}
\usepackage{graphics,color}
\usepackage{enumerate, enumitem}
\usepackage{mathtools,centernot}
\usepackage{cases}
\usepackage{amsrefs}
\usepackage{bbm}
\usepackage{xfrac}
\usepackage{xcolor}
\usepackage{autonum}

\usepackage{bookmark}

\newtheorem{theorem}{Theorem}[section]
\newtheorem{lemma}[theorem]{Lemma}

\newtheorem{corollary}[theorem]{Corollary}
\newtheorem{definition}[theorem]{Definition}
\newtheorem{proposition}[theorem]{Proposition}
\newtheorem{remark}[theorem]{Remark}

\numberwithin{equation}{section}

\newcommand{\T}{\ensuremath{\mathbb{T}}}
\newcommand*{\R}{\ensuremath{\mathbb{R}}}

\newcommand*{\N}{\ensuremath{\mathbb{N}}}

\newcommand{\eps}{\varepsilon}

\newcommand{\quotes}[1]{``#1''}

\renewcommand{\MR}[1]{} 

\usepackage{color, graphicx}
\usepackage{mathrsfs, dsfont}

\usepackage[]{hyperref}
\hypersetup{
colorlinks=true,       
linkcolor=red,          
citecolor=blue,        
filecolor=red,      
urlcolor=cyan           
}

\def\div{\mathop{\rm div}\nolimits}    
\def\curl{\mathop{\rm curl}\nolimits}    
 
\def\Lip{\mathop{\rm Lip}\nolimits}

\newcommand{\be}{\begin{equation}}
\newcommand{\ee}{\end{equation}}

\title{No anomalous dissipation in two-dimensional incompressible fluids}

\author{Luigi De Rosa}
\address{Department Mathematik Und Informatik, Universit\"at Basel, CH-4051 Basel, Switzerland}
\email{luigi.derosa@unibas.ch}

\author{Jaemin Park}
\address{Department Mathematik Und Informatik, Universit\"at Basel, CH-4051 Basel, Switzerland}
\email{jaemin.park@unibas.ch}

\date{\today}

\thanks{\textit{Acknowledgements}. L.D.R. acknowledge  the support of the SNF grant FLUTURA: Fluids, Turbulence, Advection No. 212573, while the work of J.P. has been funded by the SNSF Ambizione grant No. 216083
}

\begin{document}

\begin{abstract}
We prove that any sequence of vanishing viscosity Leray--Hopf solutions to the periodic two-dimensional incompressible Navier--Stokes equations does not display anomalous dissipation if the initial vorticity is a measure with positive singular part. A key step in the proof is the use of the Delort--Majda concentration-compactness argument to exclude formation of atoms in the vorticity measure, which in particular implies that the limiting velocity is an admissible weak solution to Euler. This is the first result proving absence of dissipation in a class of solutions in which the velocity fails to be strongly compact in $L^2$, putting two-dimensional turbulence in sharp contrast with respect to that in three dimensions. Moreover, our proof reveals that the amount of energy dissipation can be bounded by the vorticity measure of a disk of size $\sqrt \nu$, matching the two-dimensional Kolmogorov dissipative length scale which is expected to be sharp. 
\end{abstract}

\maketitle

\section{Introduction}

In the two-dimensional spatially periodic setting $\T^2\times (0,T)$ we will consider the incompressible Euler equations
\begin{equation}\label{E}
\left\{\begin{array}{l}
\partial_t u + \div (u\otimes u) +\nabla p=0 \\
\div u=0\\
u(\cdot,0)=u_0.
\end{array}\right.
\end{equation}
Recall the  notion of weak solution.
\begin{definition}[Euler weak solutions]
\label{D:E weak sol}
Let $u_0\in L^2(\T^2)$ be a given incompressible vector field. We say that $u\in L^2(\T^2\times [0,T])$ is a weak solution to \eqref{E} if $\div u=0$ and 
$$
\int_0^T \int_{\T^2} \left(u\cdot \partial_t\varphi +u\otimes u :\nabla \varphi \right) \,dxdt =-\int_{\T^2} u_0 \cdot \varphi (x,0)\,dx\qquad \forall \varphi\in C^1_c (\T^2\times [0,T)), \, \div \varphi=0.
$$
\end{definition}
In what follows we will denote by $E_u:[0,T]\rightarrow [0,\infty]$ the kinetic energy of $u$, that is 
$$
E_u(t):=\frac{1}{2}\int_{\T^2} |u(x,t)|^2\,dx.
$$
In view of the celebrated Kolmogorov Theory of Turbulence \cite{K41}, and the subsequent Onsager ideal picture \cite{O49}, in recent years a huge mathematical effort has been put in the study of weak solutions to the Euler equations which fail to conserve the kinetic energy. In this context, independently on the space dimension, we nowadays know the critical regularity to be $ L^3_tB^{\sfrac{1}{3}}_{3,\infty}$: if $u \in L^3_tB^{\sfrac{1}{3}+}_{3,\infty}$ then the kinetic energy has to be conserved \cites{E94,CET94,CCFS08,MR4691856} while having $u\in L^3_tB^{\sfrac{1}{3}-}_{3,\infty}$ is in general compatible with energy dissipation \cites{Is18,BDSV19,GR23}. The latter works were built on the convex integration techniques introduced by De Lellis and Sz\'ekelyhidi \cite{DS13} in the context of incompressible turbulent flows. 

Remarkably, as it has been first noted in \cite{CLLS16}, the two-dimensional case is quite special if we restrict to solutions of Euler arising as vanishing viscosity limits (such solutions are often called \emph{physically realizable solutions}), and in particular the Onsager critical scaling can be overcome. In order to give the precise statements we recall the incompressible Navier--Stokes system with viscosity $\nu>0$
\begin{equation}\label{NS}
\left\{\begin{array}{l}
\partial_t u^\nu + \div (u^\nu\otimes u^\nu) +\nabla p^\nu=\nu \Delta u^\nu \\
\div u^\nu =0\\
u^\nu(\cdot,0)=u^\nu_0,
\end{array}\right.
\end{equation}
on $\T^2\times (0,T)$. Here the natural class of weak solutions is the so-called Leray--Hopf solutions.

\begin{definition}[Leray--Hopf weak solutions]\label{LHsol}
Let $u^\nu_0\in L^2(\T^2)$ be a given incompressible vector field and $\nu>0$. We say that $u^\nu \in L^\infty ([0,T];L^2(\T^2))\cap L^2([0,T];H^1(\T^2))$ is a Leray--Hopf weak solution to \eqref{NS} if $\div u^\nu =0$,
\begin{align}\label{weak_formulation_NS}
\int_0^T \int_{\T^2} \left(u^\nu \cdot \partial_t\varphi +u^\nu\otimes u^\nu  :\nabla \varphi + \nu u^\nu \cdot \Delta \varphi\right) \,dxdt =-\int_{\T^2} u^\nu_0 \cdot \varphi (x,0)\,dx,
\end{align}
holds for any $\varphi\in C^1_c (\T^2\times [0,T))$ such that $ \div \varphi=0$ and in addition 
\begin{equation}\label{Leray en ineq}
E_{u^\nu} (t)+ \nu \int_0^t \int_{\T^2} |\nabla u^\nu (x,s)|^2 \,dxds\leq E_{u^\nu_0} \qquad \text{for a.e. } t\in [0,T].
\end{equation}
\end{definition}
The existence of such solutions is due to Leray \cite{L34} (later on refined by Hopf \cite{Hopf51}), whose strategy directly applies to both the two and the three-dimensional case. It is well known \cite{RR} that in the two-dimensional setting Leray--Hopf solutions are unique and become instantaneously smooth, i.e. $u^\nu\in C^\infty(\T^2\times (0,T])$. Moreover, the energy inequality \eqref{Leray en ineq} upgrades to an equality. Remarkably, non-uniqueness in $3$ dimensions seems to be expected \cites{JS14,JS15} and recently fully proved with an additional forcing term \cite{ABC22}.

In this setting, perhaps the only physically relevant one, the term \emph{anomalous dissipation} refers to 
\begin{equation}\label{anom dissip}
\liminf_{\nu\rightarrow 0} \nu \int_0^T \int_{\T^2} |\nabla u^\nu (x,t)|^2 \,dxdt>0.
\end{equation}
Such phenomenon has a quite strong experimental and numerical evidence and we refer the interested reader to the monograph \cite{Frisch95} for an extensive physical and historical overview on the theory of turbulent flows. Our purpose in this paper is to rule out the possibility of \eqref{anom dissip} even when the initial data is sufficiently rough in the two-dimensional setting.

In view of the energy balance \eqref{Leray en ineq}, it is evident that, as soon as $\{u_0^\nu\}_\nu$ is bounded in $L^2$, the kinetic energy is always bounded independently of $\nu$. This allows us to extract a subsequence of $\{u^{\nu}\}_\nu$ which converges weakly to some $u$ in $L^\infty_tL^2_x$. For such a sequence, one can easily deduce  that $u$ must satisfy $E_u(t)\le E_{u_0}$ for a.e. $t$. In a common terminology such solutions are called \emph{admissible}.
\begin{definition}[Admissible Euler weak solutions]\label{D: E admiss sol}
Let $u_0\in L^2(\T^2)$ be a given incompressible vector field. We say that a weak solution $u$ to \eqref{E} is admissible if 
\begin{equation}
\label{amiss cond}
E_u(t) \leq E_{u_0} \qquad \text{for a.e. } t\in [0,T]. 
\end{equation}
In particular $u\in L^\infty([0,T];L^2(\T^2))$.
\end{definition}
However even if anomalous dissipation does not happen, it is not trivial at all  whether the inequality in \eqref{amiss cond} can be  replaced by equality, unless $\{u^\nu\}_\nu$ converges strongly in $L^2_{x,t}$. Indeed, in general, the absence of anomalous dissipation does not necessarily guarantee the conservation of the energy of the limit. A version of a converse implication follows trivially; if $u$ is a solution to Euler arising from vanishing viscosity and it conserves the energy, then there is no anomalous dissipation. However before investigating the connection between  energy conservation and anomalous dissipation, one should keep in mind  that justifying whether $u^\nu\to u$ solves the Euler equations is highly nontrivial in any setting in which strong $L^2_{x,t}$ compactness is not accessible.  We first present the available results regarding weak solutions to the Euler equations, and then discuss relevant findings concerning the energy conservation of these solutions.

\subsection{Weak solution \&  energy conservation for Euler.}\label{ifyoufindthis_emailthe_authors}
 Apart from  kinetic energy, the two-dimensional fluids hold another conserved quantity: the vorticity. Throughout this whole note, we will denote  the vorticity  by $\omega^\nu:=\curl u^\nu$.  
 
  In the literature, weak solutions to the Euler equations have been constructed under additional assumptions on the initial vorticity $\omega_0$. For $\omega_0\in L^\infty$, Yudovich \cite{yudovich1963non} proved existence and uniqueness of global weak solutions. Afterwards, the existence result was extended for $\omega_0\in L^p$ with $p>1$ by DiPerna and Majda \cite{diperna1987concentrations}.  Remarkably, Delort \cite{delort1991existence} constructed global weak solutions for initial data of vortex sheet type, that is, $\omega_0\in \mathcal{M}$, where $\mathcal{M}$ denotes the space of Radon measures, assuming that the singular part of $\omega_0$ has a distinguished sign.  We also refer to \cites{evans1994hardy,vecchi19931,scho95} for different proofs and further extensions. On the opposite side, the $L^p$ non-uniqueness in the presence of an external forcing term has been proved by Vishik in the seminal papers \cites{V1,V2}. We also refer to the recent result \cite{FM23} proving sharpness of the Yudovich's assumptions for (unforced) weak-strong uniqueness.

 Concerning the energy conservation of two-dimensional inviscid fluids,  the Onsager criterion $u\in B^{\sfrac{1}{3}}_{3,\infty}$ can be translated into the integrability of the vorticity, $\omega\in L^{\sfrac{3}{2}}$, noting that these two spaces share the same scaling invariance (these spaces often called \emph{Onsager critical spaces}). Indeed, the aforementioned result in \cite{CCFS08} immediately implies that if $\omega\in L^\infty_{t}L_x^{\sfrac{3}{2}}$, then the energy must be conserved whether the solution is coming from vanishing viscosity or not. However vanishing viscosity solutions certainly exhibit a special feature.  In \cite{CLLS16}, it has been proved (see also \cite{LMP21} for further improvements) that, if $\omega_0\in L^p$ for $p>1$, then anomalous dissipation  cannot happen. Moreover, since $p>1$, the sequence of velocities $\{u^\nu\}_\nu$ is strongly compact in $L^2_{x,t}$, which in particular implies, up to possibly considering a subsequence, that the limit is a weak solution to Euler with constant kinetic energy. See for instance the simple \cref{P:en cons and compact} (or \cite{LMP21}*{Theorem 2.11} for a more general argument) for the connection between energy conservation and the strong convergence of velocity.
  
  \subsection{Main result and remarks}
Our main result establishes that anomalous dissipation cannot occur not only in the full range $p\ge 1$ but also when considering measure initial vorticities, provided their singular part holds distinguished sign. To be precise, we prove the following
\begin{theorem}\label{T:main}
Let $\{u^\nu_0\}_{\nu>0} \subset L^2(\T^2)$ be a sequence of divergence-free vector fields satisfying
\begin{enumerate}[label=(H\arabic*)]
\item\label{H1} $\left\{u_0^\nu\right\}_{\nu>0}$ is strongly compact in $L^2(\mathbb{T}^2)$;
\item\label{H2}   $\{\omega^\nu_0\}_{\nu>0}$ is bounded in $\mathcal{M}(\mathbb{T}^2)$ and it  admits a decomposition $\omega_0^\nu =f^\nu_0 + \Omega_0^\nu$ such that  $\left\{ f^\nu_0\right\}_{\nu>0}$ is weakly compact in $L^1(\mathbb{T}^2)$ and  $\Omega^\nu_0 \geq 0$.
\end{enumerate}
Let $\{u^\nu\}_{\nu>0}$ be the corresponding sequence of Leray--Hopf solutions to \eqref{NS} with $\nu\to 0$. Then 
\begin{enumerate}[label=(T\arabic*)]
\item\label{T1} there exists a subsequence $\{u^{\nu_n}\}_n$ such that $u^{\nu_n} \rightharpoonup^* u$ in $L^\infty([0,T];L^2(\T^2))$, $u^{\nu_n}_0\rightarrow u_0$ in $L^2(\T^2)$ and $u$ is an admissible weak solution to Euler with initial datum $u_0$. In addition, if $\omega:=\curl u$, it holds 
$$
\|\omega(t)\|_{\mathcal{M}(\T^2)}\leq \sup_{\nu>0}\|\omega^\nu_0\|_{\mathcal{M}(\T^2)}\qquad \text{for a.e. } t\in [0,T];
$$
\item\label{T2} the full sequence does not display anomalous dissipation, i.e.
\begin{equation}
    \label{no_an_diss_intro}
    \limsup_{\nu\rightarrow 0} \nu\int_0^T\int_{\T^2} |\nabla u^\nu (x,t)|^2\,dxdt=0.
\end{equation}
\end{enumerate}
\end{theorem}
Some remarks are in order. First, the proof of \ref{T1} relies on the celebrated concentration-compactness result by Delort \cite{delort1991existence} in the full space $\R^2$. The original proof by Delort deals with a direct regularization of the initial datum in the Euler equations. Then, still in $\R^2$,  Majda \cite{majda1993remarks} proved  that initial vorticities of distinguished sign $\omega^\nu_0\geq 0$ allow to construct admissible weak solutions to Euler from vanishing viscosity sequences of Leray--Hopf solutions to the Navier--Stokes equations. In this context, our claim \ref{T1} is nothing but a slight extension of such results since having a vorticity with distinguished sign would not be compatible with being in our setting (the argument in \cite{majda1993remarks} fails due to sign-changing vorticity which is unavoidable in periodic setting), and the proof in \cite{delort1991existence} is not directly adaptable in the vanishing viscosity context. However, the essence of all proofs, including ours presented in this paper, lies in the phenomenon that the absolute vorticity $|\omega^\nu|$ does not develop atoms  as $\nu\to 0$, which was first observed by Delort \cite{delort1991existence}. For this reason we may refer to this part of the proof as a \quotes{Delort--Majda} argument, which is the aim of Section~\ref{D_M_argu}.  We also refer to \cite{Gerard92} for a description/extension of the Delort ideas, with particular emphasis on the very last remark of the work.

Second, our proof of \ref{T2} is divided in two parts: we first prove that there is no anomalous dissipation in time interval $(\delta,T)$ for any $\delta>0$, and then that the same holds in $(0,\delta)$. The first part relies on a careful quantitative estimate \eqref{sym2} on the $L^2$ norms of $\omega^\nu$ and its gradient $\nabla \omega^\nu$ which reduces the problem to show that no atomic vorticity concentration can happen in the sequence $\{ |\omega^\nu|\}_\nu$. Then,  prohibiting anomalous dissipation to happen in $(0,\delta)$ is a consequence of the Delort--Majda argument of \ref{T1}, which ensures that the kinetic energy of any weak limit of $\{u^\nu\}_\nu$ is continuous from the right in $t=0$ (see \cref{L:energy right cont}). Let us also specify that in this second step the strong compactness of the initial velocities in \ref{H1} is used in the argument to conclude the absence of dissipation in $(0,\delta)$, which fails if only weak $L^2$ compactness is assumed. However, as stated in \ref{H2}, no strong compactness of the initial vorticities is necessary. The proof of \ref{T2} will be given in Section~\ref{NAD}. Some additional remarks about the proof, together with an extension of our main result, will be then given in \cref{S:further remarks}.

\begin{remark}A physically interesting feature arising in the proof of \cref{T:main} is that the total dissipation can be bounded by the amount of vorticity measure contained in a disk of size proportional to $\sqrt{\nu}$, which coincides with the dissipative length scale of Kolmogorov in the two-dimensional setting (see for instance \cite{Dav_turb}). Thus, the bounds obtained in our proof are presumably sharp. Further remarks about the appearance of the Kolmogorov dissipative length scale will be discussed in \cref{S:further remarks}.
\end{remark}

We emphasize that our theorem does not rely on strong $L^2$ compactness of the Leray--Hopf solutions. As proved in \cite{LMP21}*{Theorem 2.11}, a compactness of $\left\{ u^\nu\right\}_\nu$ in $L^2$ is equivalent to the the energy conservation of the limit $u=\lim_{\nu\to 0}u^\nu$. If $\{\omega^\nu\}_\nu$ stays bounded in $ L^p$ for some $p>1$, standard embedding theorems ensure that $\left\{u^\nu\right\}_\nu$ is strongly compact in $L^2$. However to the best of our knowledge, the energy conservation in the vanishing viscosity limit remains an open question for any space of vorticity in which $L^2$ strong compactness of $\left\{ u^\nu\right\}_\nu$ is not guaranteed (e.g. $\omega\in L^1$).  In our setting, the vorticity is merely a measure, and  our theorem seems to be the first one proving the absence of dissipative anomaly in a class of solutions which fails to compactly embed in $L^2$. {This provides a sharp difference between two and three dimensional turbulent flows, since in the latter there seems to be no hope to disprove anomalous dissipation without strong compactness.}  Note also that, in general, weak solutions to Euler with initial vorticity $\omega_0\not \in L^1$ do not conserve the kinetic energy \cite{BC_lorenz}, even if they are admissible \cite{Modena_hardy}. To conclude, we believe to be quite interesting that it is possible to prove no anomalous dissipation in the full range in which global weak solutions to Euler are known to exist. 

\subsection*{Data Availability \& Conflict of Interest Statements} Data sharing not applicable to this article as no datasets were generated or analysed during the current study.  All authors declare that they have no conflicts of interest.


\section{Auxiliary tools}
Here we collect some standard tools which will be used in this work. First let us specify several conventional notations.

\subsection{Notations} We denote by $C$ a universal positive constant which may vary from line to line. In case such an implicit constant depends on a quantity, let say $A$, and its dependence is necessary to be denoted, we will write $C_A$.  

For any $1\leq p\leq \infty$ we will use the standard notation $L^p(\Omega)$ and $\mathcal{M}(\Omega)$ to denote the classical spaces of $p$-integrable functions and finite Radon measures respectively, over a measurable set $\Omega$. The respective norms will be denoted by $\| \cdot \|_{L^p}$ and $\| \cdot \|_{\mathcal{M}}$. For any time-dependent function with values in a Banach space $Y$ we will write 
\begin{align}
f:[0&,T]\rightarrow Y\\
&t\mapsto f(t).
\end{align}
The corresponding norm in $Y$ for a fixed $t$-time slice will be denoted by $\|f(t)\|_{Y}$. In case any ambiguity might arise, we will write $f_t$ instead of $f(t)$. Consequently, the time dependent Bochner space $L^p([0,T];Y)$ is the set of strongly measurable $f:[0,T]\rightarrow Y$ for which the real valued function $t\mapsto \|f(t)\|_Y$ belongs to $L^p([0,T])$.

\subsection{Weak compactness criterion}
 A set of integrable functions $\mathcal{F}\subset L^1(\mathbb{T}^2)$ is said to be equi-integrable if for any $\eps>0$, there exists $\delta>0$ such that
\[
|A|< \delta\quad  \implies \quad  \sup_{f\in \mathcal{F}}\int_{A}|f(x)|\,dx < \eps.
\]
Here $A\subset \T^2$ is any Borel set and $|A|$ denotes its Lebesgue measure. The Dunford-Pettis theorem guarantees that $\mathcal{F}\subset L^1(\mathbb{T}^2)$ is weakly compact in $L^1$ if and only if $\mathcal{F}$ is equi-integrable. The following lemma provides another criterion concerning the weak $L^1$ topology. 
\begin{lemma}[\cite{K08}*{Theorem 6.19}]\label{Klenke} A set of functions $\mathcal{F}\in L^1(\mathbb{T}^2)$ is equi-integrable if and only if there exists a convex even function $\beta:\R\rightarrow [0,\infty)$, monotone increasing on $[0,\infty)$, and with $\lim_{s\to \pm \infty}\sfrac{\beta(s)}{|s|}= +\infty$ such that
\[
\sup_{f\in \mathcal{F}}\int_{\mathbb{T}^2}\beta(f(x))\,dx <\infty.
\]
\end{lemma} 
The function $\beta$ can be chosen to be smooth by a standard approximation argument.

\subsection{Remarks on the Euler and Navier--Stokes equations} Here we recall some well known facts about weak solutions to the incompressible Euler equations. We start with the following.
\begin{remark}[Admissible solutions are $C^0_tL^2_{\rm w}$]\label{R:admiss are cont}
Let $u$ be an admissible weak solution to Euler in the sense of \cref{D: E admiss sol}. It is well known that up to redefine $u$ on a Lebesgue negligible set of times we have in addition $u\in C^0([0,T];L^2_{\rm w} (\T^2))$, where $L^2_{\rm w}$ denotes the space of $L^2$ vector fields endowed with the usual weak topology. For a proof of this statement see for instance \cite{DS10}*{Lemma 2.2} and the proof of \cite{DS10}*{Lemma 7.1}. In particular, by taking the continuous in time representative, the kinetic energy is well-defined at every time $t$ and by lower semicontinuity of the $L^2$ norm under weak convergence we can upgrade \eqref{amiss cond} to hold for all times $t\in [0,T]$.
\end{remark}
Then, it is a straightforward consequence of the above remark that the kinetic energy of $u$ becomes right continuous in $t=0$.
\begin{lemma}
\label{L:energy right cont}
Let $u_0\in L^2(\T^2)$ be a given divergence-free initial datum and $u\in C^0([0,T];L^2_{\rm w} (\T^2))$ be an admissible weak solution to \eqref{E} according to \cref{D: E admiss sol}. Then its kinetic energy $E_u$ is continuous from the right at $t=0$.
\end{lemma}
\begin{proof}
Since $u\in C^0([0,T];L^2_{\rm w} (\T^2))$, by lower semicontinuity of the $L^2$ norm under weak convergence we have
$$
E_{u_0} \leq \liminf_{t_n\rightarrow 0} E_{u}(t_n).
$$
By \cref{R:admiss are cont}, the admissibility condition \eqref{amiss cond} also implies $E_u (t_n) \leq E_{u_0}$ for all $n\in \N$, from which we get
$$
E_{u_0} \leq \liminf_{t_n\rightarrow 0} E_{u}(t_n)\leq \limsup_{t_n\rightarrow 0} E_u (t_n)\leq E_{u_0},
$$
that is $\lim_{t_n\rightarrow 0} E_u(t_n)=E_{u_0}$.
\end{proof}

We conclude this section by recalling the following classical estimate for the two-dimensional Navier--Stokes equations. We include the proof for the reader's convenience.

\begin{proposition}
    \label{P:vort_est_2d}
    Let $\{u_0^\nu\}_{\nu>0}\subset L^2(\T^2)$ be divergence-free initial data and let $\{u^\nu\}_{\nu>0}$ be the corresponding sequence of Leray--Hopf solutions. Then 
    $$
\,\|\omega^\nu(t)\|_{L^2}\leq \frac{\|u^\nu_0\|_{L^2}}{\sqrt{2t \nu}}\qquad \forall t>0.
$$
\end{proposition}
\begin{proof}
    By the enstrophy balance on the vorticity formulation
$$
\frac{d}{dt} \|\omega^\nu(t)\|^2_{L^2} =- 2\nu \|\nabla \omega^\nu(t)\|^2_{L^2} \qquad \text{for all } t>0,
$$
we deduce that $t\mapsto \|\omega^\nu(t)\|^2_{L^2}$ is monotone non-increasing. Thus, the energy inequality \eqref{Leray en ineq}, together with $\|\omega^\nu(t)\|_{L^2}=\|\nabla u^\nu(t)\|_{L^2}$, implies
$$
t \nu\,\|\omega^\nu(t)\|^2_{L^2}\leq \nu \int_0^t \|\omega^\nu(s)\|^2_{L^2}\,ds=\nu \int_0^t \|\nabla u^\nu(s)\|^2_{L^2}\,ds\leq \frac{\|u^\nu_0\|^2_{L^2}}{2}.
$$
\end{proof}

\subsection{Advection-diffusion equation with $L^1$ initial data} Here we prove some useful quantitative estimates and weak compactness results for advection-diffusion equations with $L^1$ initial data.

We consider a  passive scalar $\theta^{\nu}\in C^\infty(\mathbb{T}^2\times [0,T))$ which solves
\begin{equation}\label{AD}
\left\{\begin{array}{l}
\partial_t\theta^\nu+ v^{\nu}\cdot\nabla \theta^\nu = \nu \Delta \theta^\nu \\
\theta^\nu(\cdot ,0)=\theta^\nu_0\in C^\infty(\mathbb{T}^2)
\end{array}\right.
\end{equation}

Throughout this section, we will assume that the velocity $v^\nu$ is smooth in $(x,t)$ and $\div v^\nu=0$. For the initial data, we will assume that 
\begin{align}\label{initial_scalar}
\sup_{\nu>0}\| \theta^\nu_0\|_{L^1}<\infty,
\end{align}
while no quantitative property is assumed on the velocity $v^\nu$.

By using the incompressibility condition of the vector field, one can easily see that  
 \begin{align}\label{uniform_bounds_1}
\| \theta^\nu(t)\|_{L^p}\le \| \theta^\nu_0\|_{L^p}\le \sup_{\nu>0} \| \theta_0^\nu\|_{L^p}, \qquad \forall t\ge0 \text{ and } \forall p\in [1,\infty].
\end{align}
In the following proposition, we derive slightly finer properties of the solutions. 
\begin{proposition}\label{equi_integrability}
Under the assumption \eqref{initial_scalar}, the following holds true.
\begin{enumerate}[label=(\arabic*)]
\item\label{worsttrip} If $\int_{\T^2} \theta_0^\nu =0$, the sequence of solutions $\left\{\theta^\nu\right\}_{\nu>0}$  satisfies
\begin{align}
\nu\int_{\delta}^T \| \theta^\nu(t)\|_{L^2}^2 \,dt\le C_{\delta,T}\sup_{\nu>0}\| \theta_0^\nu\|_{L^1}^2,\label{theta_es_fasnact1}\\
\nu^2\int_{\delta}^{T}\|\nabla\theta^\nu(t)\|_{L^2}^2 \,dt \le C_{\delta,T}\sup_{\nu>0}\| \theta^\nu_0\|_{L^1}^2,\label{theta_es_fasnact2}
\end{align}
for any $\delta>0$.
\item\label{Ilostairpods}
Assume, in addition, that $\left\{ \theta^\nu_0\right\}_{\nu>0}$ is weakly compact in $L^1(\T^2)$. Then the set of functions $\left\{ \theta^\nu(t)\right\}_{\nu>0}$ is equi-integrable, uniformly in time. More precisely, for any $\eps>0$, there exists $\delta>0$ such that if $A\subset \mathbb{T}^2$ satisfies $|A|< \delta$, then
\[
\sup_{t\ge 0,\, \nu>0}\int_{A}|\theta^\nu(x,t)| \,dx< \eps.
\]
\end{enumerate} 
\end{proposition}
We remark that the zero-average assumption on the initial data in \ref{worsttrip} is not necessary and it has been made for practical reasons in order to deal with a cleaner version of the Gagliardo-Nirenberg inequality on $\T^2$, which prevents lower order terms to appear in the estimates.
\begin{proof}
We give a proof for each item of the proposition separately.

\textsc{Proof of \ref{worsttrip}.} The proof is based on the argument in \cite{CLLS16}*{Section 3} with a slight modification. A standard $L^2$ estimate for the advection-diffusion equation \eqref{AD} yields
\begin{align}\label{energe_estimate}
\frac{1}{2}\frac{d}{dt}\| \theta^\nu(t)\|_{L^2}^2 = - \nu \| \nabla \theta^\nu(t)\|_{L^2}^2.
\end{align}
Note that since we assumed the initial data to have zero average, it follows that $\{\theta^\nu(t)\}_\nu$ stays average-free for all times $t>0$. Thus, by the Gagliardo-Nirenberg interpolation inequality we get
\[
\|  \theta^\nu(t)\|_{L^2}\le C \| \theta^\nu(t) \|_{L^1}^{\sfrac{1}{2}}\| \nabla \theta^\nu(t)\|_{L^2}^{\sfrac{1}{2}},
\]
yielding that $\| \nabla \theta^\nu(t)\|_{L^2}^2\ge C\| \theta^\nu(t)\|_{L^2}^{4}\| \theta^{\nu}(t)\|_{L^1}^{-2}$. Plugging this into \eqref{energe_estimate}, we see that
\begin{align}
\frac{1}{2}\frac{d}{dt}\| \theta^\nu(t)\|_{L^2}^2 &\le- C\nu \| \theta^\nu(t)\|_{L^1}^{-2}\| \theta^\nu(t)\|_{L^2}^{4} \le - C\nu \| \theta^\nu(t)\|_{L^2}^{4}\sup_{\nu>0}\| \theta^\nu_0\|_{L^1}^{-2},
\end{align}
where the last inequality is due to \eqref{uniform_bounds_1}.
Denoting by
\[
M:=\sup_{\nu>0} \| \theta^\nu_0\|_{L^1}^{2}<\infty \qquad \text{ and } \qquad  y(t):=\| \theta^{\nu}(t)\|_{L^2}^2,
\]
the above estimate for $\frac{d}{dt}\| \theta^\nu(t)\|_{L^2}^2$ can be simply written as $y'(t)\le -CM^{-1} \nu y(t)^{2}$.  Using Gr\"onwall's inequality in $\left[\sfrac{\delta}{2},T\right]$ we deduce that
 $$
 \frac{1}{y(t)}\geq   C M^{- 1}\nu \left( t-\frac{\delta}{2}\right).
 $$
 Hence, integrating this over $t\in [\delta,T]$ for an arbitrary $\delta>0$, we arrive at
\[
\int_\delta^{T} y(t)\,dt \le C_{\delta,T} M \nu^{-1}= C_{T,\delta}\nu^{-1}\sup_{\nu>0}\| \theta^\nu_0\|_{L^1}^2 .
\]
Multiplying by $\nu$ both sides, we obtain \eqref{theta_es_fasnact1}. In order to see \eqref{theta_es_fasnact2}, we integrate \eqref{energe_estimate} over the time interval $[t,T]$ for arbitrary $0<t<T$, yielding that
\[
\| \theta^\nu(T)\|_{L^2}^2 + 2 \nu\int_{t}^T \| \nabla \theta^\nu(s)\|_{L^2}^2 \,ds = \| \theta^\nu(t)\|_{L^2}^2.
\]
Hence, we have  $ \nu\int_{t}^T \| \nabla \theta^\nu(s)\|_{L^2}^2 \,ds\le \| \theta^\nu(t)\|_{L^2}^2$  for all $0<t<T$. We again integrate this in $t$ over $[\sfrac{\delta}{2},T]$ for an arbitrary $\delta>0$, which gives us
\begin{align}\label{nu_1}
\nu\int_{\sfrac{\delta}{2}}^{T}\int_{t}^{T}\| \nabla \theta^\nu(s)\|_{L^2}^2 \,ds dt \le \int_{\sfrac{\delta}{2}}^{T}\| \theta^{\nu}(t)\|_{L^2}^2 \,dt\le C_{\delta,T}\nu^{-1}\sup_{\nu>0}\| \theta^\nu_0\|_{L^1}^2,
\end{align}
where the last inequality is due to \eqref{theta_es_fasnact1}. Using Fubini's theorem, the left-hand side of this can be estimated as
\begin{align}
\nu\int_{\sfrac{\delta}{2}}^{T}\int_{t}^{T}\| \nabla \theta^\nu(s)\|_{L^2}^2 \,ds dt& = \nu\int_{\sfrac{\delta}{2}}^T\| \nabla\theta^\nu(s)\|_{L^2}^2\left(\int_{\sfrac{\delta}{2}}^{s}\,dt\right)ds\\
&= \nu\int_{\sfrac{\delta}{2}}^T\| \nabla\theta^\nu(s)\|_{L^2}^2\left(s-\frac{\delta}{2}\right)ds\\
&\ge \nu\int^{T}_{\delta}\| \nabla\theta^\nu(s)\|_{L^2}^2\left(s-\frac{\delta}{2}\right)ds\\
&\ge \nu \frac{\delta}{2} \int^{T}_{\delta}\| \nabla\theta^\nu(s)\|_{L^2}^2\,ds.
\end{align}
Combining this with \eqref{nu_1}, we obtain the estimate \eqref{theta_es_fasnact2}.

\textsc{Proof of \ref{Ilostairpods}.} Thanks to Lemma~\ref{Klenke} and the weak $L^1$ compactness of $\left\{\theta^\nu_0\right\}_\nu$, we find a convex function $\beta:\R\rightarrow[0,\infty)$ such that 
\begin{align}\label{weak_com}
\sup_{\nu>0}\int_{\mathbb{T}^2}\beta(\theta^\nu_0(x))\,dx <\infty. 
\end{align}
Using the advection-diffusion equation~\eqref{AD} and the incompressibility of $v^\nu$, we see that
\[
\frac{d}{dt}\int_{\mathbb{T}^2}\beta(\theta^\nu(x,t))\,dx = -\nu\int_{\mathbb{T}}\beta''(\theta^\nu(x,t))|\nabla \theta^\nu(x,t)|^2 \,dx\le 0,
\]
where the last inequality follows from the convexity of $s\mapsto \beta(s)$. Integrating in time, we get $\int_{\mathbb{T}^2}\beta(\theta^{\nu}(x,t))\,dx \le \int_{\mathbb{T}^2}\beta(\theta^{\nu}_0(x))\,dx$. Taking the supremum in $t,\nu$, we arrive at
\[
\sup_{t\ge 0,\,\nu>0}\int_{\mathbb{T}^2}\beta(\theta^{\nu}(x,t))\,dx \le \int_{\mathbb{T}^2}\beta(\theta^{\nu}_0(x))\,dx<\infty.
\]
Again applying Lemma~\ref{Klenke}, we obtain  the equi-integrability of $\left\{\theta^\nu\right\}_{\nu}$.
 \end{proof}


\section{A Delort--Majda argument}\label{D_M_argu}
In this section  we study the the behaviour as $\nu\rightarrow 0$ of Leray--Hopf solutions to \eqref{NS}, which will be denoted by $\{u^\nu\}_\nu$,  with initial data $\{u_0^\nu\}_\nu \subset L^2(\T^2)$ whose vorticity is a Radon measure. We recall the conventional notations  $\omega_0^\nu:=\curl u^\nu_0$ and $\omega^\nu:=\curl \omega^\nu$. 

In the first part of this section, we  investigate some properties of the Leray--Hopf solutions with the initial data satisfying \ref{H1} and \ref{H2}. Afterwards, we re-prove the existence of global admissible weak solutions to the Euler equation~\eqref{E} obtained by Delort \cite{delort1991existence} in the context of vanishing viscosity limit, which in turn generalizes Majda's proof \cite{majda1993remarks} in the case in which the initial vorticity has an $L^1$ part with non-distinguished sign.

\subsection{Leray--Hopf solutions to the Navier--Stokes equations} 
We denote $\rho_\alpha$ a usual mollifier, that is,
\begin{align}\label{mollifier_1}
\rho_\alpha(x):=\frac{1}{\alpha^2}\rho(\alpha^{-1}x), \qquad \text{ for some $\rho\in C^\infty_c(B_1(0))$, \,\, with } \int_{\mathbb{R}^2} \rho \,dx=1.
\end{align}
where $B_R(x)$ denotes the disk centered at $x$ with radius $R$.

 Our goal is to derive some estimates on $\{u^\nu\}_\nu$, uniform in viscosity, depending only on $\sup_{\nu>0} \|\omega_0^\nu\|_{\mathcal M}$. Note that, since the initial velocities are only assumed to be $L^2$, the sequence of Leray--Hopf solutions $\{u^\nu\}_\nu$, despite smooth for all strictly positive times, fails to be smooth up to $t=0$. Thus, to avoid getting into technicalities about renormalization properties of advection-diffusion equations with non-smooth vector fields, we will mollify the initial data and prove the corresponding estimates for the regularized solutions, which will be uniform in the mollification parameter. 
 
 We denote $ u^\nu_{0,\alpha}:=u^\nu_0*\rho_\alpha$ for some $\alpha>0$. Let $u^\nu_\alpha$ be the corresponding Leray--Hopf  solution to the Navier--Stokes equation with the initial data $u^\nu_{0,\alpha}$ and denote by $\omega^\nu_{0,\alpha}:=\curl u^\nu_{0,\alpha}$ and $\omega^\nu_\alpha:=\curl u^\nu_\alpha$. Using the properties of the mollifier, it immediately follows that
 \begin{align}\label{intial_bounds}
 \| u_{0,\alpha}^{\nu}\|_{L^2} \le \sup_{\nu>0}\| u_0^\nu\|_{L^2}\qquad \text{ and } \qquad  \| \omega^\nu_{0,\alpha}\|_{L^1}\le \sup_{\nu>0}\| \omega^\nu_0\|_{\mathcal{M}}.
 \end{align} 
 Since $u^\nu_{0,\alpha}$ is a smooth divergence-free vector field, the classical theory of the Navier--Stokes equation guarantees that $u^\nu_\alpha$ exists and remains smooth globally in time. Then the vorticity formulation of the Navier--Stokes equations reads
\[
\left\{\begin{array}{l}
\partial_t\omega^{\nu}_{\alpha} + u^{\nu}_{\alpha}\cdot\nabla \omega^{\nu}_\alpha = \nu \Delta \omega^\nu_\alpha \\
\omega^\nu_\alpha(\cdot, 0)=\omega^\nu_{0,\alpha}.
\end{array}\right.
\]
Using the bound of the initial vorticity in \eqref{intial_bounds}, it is trivial to see that
\begin{align}\label{L1estimate_1}
\| \omega^\nu_\alpha(t)\|_{L^1}\le \sup_{\nu>0}\| \omega^\nu_{0}\|_{\mathcal{M}}.
\end{align} 
Applying \eqref{theta_es_fasnact1} and \eqref{theta_es_fasnact2}  we also obtain
\begin{align}\label{omega_alpha}
\nu\int_{\delta}^{T}\| \omega^\nu_\alpha(t)\|_{L^2}^2 \,dt + \nu^2\int_\delta^{T}\| \nabla \omega^\nu_\alpha(t)\|_{L^2}^2\, dt \le C_{\delta,T}\| \omega^\nu_{0,\alpha}\|_{L^1}^2\le C_{\delta,T}\sup_{\nu>0}\| \omega^\nu_0\|_{\mathcal{M}}, 
\end{align}
for any $\delta>0$. Being a Leray--Hopf solution,  $u^\nu_\alpha$ will satisfy the energy balance \eqref{Leray en ineq}. Thus it is bounded in $L^\infty ([0,T];L^2 (\T^2))\cap L^2 ([0,T];H^1(\T^2))$ uniformly in $\alpha>0$ (for a fixed $\nu>0$).  Therefore, we can find a  subsequence of $\left\{ u_\alpha^\nu\right\}_\alpha$ that is converging to some $v^\nu$ strongly in $L^2 (\T^2\times [0,T])$, and also weakly  in $L^\infty ([0,T];L^2 (\T^2))\cap L^2 ([0,T];H^1(\T^2))$, as $\alpha \to 0$. Therefore, $v^\nu$ solves \eqref{weak_formulation_NS} with initial datum $u^\nu_0$ and by lower semicontinuity of the corresponding norms under weak convergence, $v^\nu$ is also a Leray--Hopf solution. By uniqueness $u^\nu=v^\nu$.  Then, taking $\alpha\to 0$ in \eqref{L1estimate_1} and  \eqref{omega_alpha}, we arrive at the following statement.
\begin{lemma}\label{estimate_lem}
The sequence of Leray--Hopf solutions $\{u^\nu\}_{\nu>0}$ enjoys 
\begin{align}\label{sym1}
\sup_{t\geq 0} \| \omega^\nu(t)\|_{L^1}&\le \sup_{\nu>0}\| \omega^\nu_0\|_{\mathcal{M}} 
\end{align}
and
\begin{align}\label{sym2}
\nu\int_{\delta}^{T}\| \omega^\nu(t)\|_{L^2}^2 \, dt + \nu^2\int_\delta^{T}\| \nabla \omega^\nu(t)\|_{L^2}^2 \, dt &\le  C_{\delta,T}\sup_{\nu>0}\| \omega^\nu_0\|_{\mathcal{M}}^2,\qquad \forall \delta>0.
\end{align}
\end{lemma}

In the next proposition, we  derive finer properties of the Leray--Hopf solutions concerning the vanishing viscosity limit.   
\begin{proposition}\label{LH_sol_properties}
Let $\{u^\nu_0\}_{\nu>0}$ satisfy \ref{H1} and \ref{H2}. Then the vorticity $\omega^\nu$ of the Leray--Hopf solution $u^\nu$ satisfies the following.
\begin{enumerate}[label=(\arabic*)]
\item\label{equi_LH}
$\omega^\nu$ admits a decomposition, $
\omega^\nu=f^\nu + \Omega^\nu$
such that $\Omega^\nu\ge 0$ and 
\begin{equation}\label{jaemin loves kimchi}
\sup_{t\ge0}\| f^\nu(t)\|_{L^1}\le \sup_{\nu>0}\| f^\nu_0\|_{L^1}.
\end{equation}
Moreover, $\{f^\nu(t)\}_{\nu>0}$ is equi-integrable, uniformly in time. More precisely, for any $\eps>0$, there exists $\delta>0$ such that
\begin{equation}\label{last equi int}
|A|< \delta\quad \implies \quad \sup_{t\ge 0, \nu\ge 0}\int_{A}|f^{\nu}(x,t)| \,dx< \eps.
\end{equation}

\item\label{no_atom_b} Up to possibly considering a subsequence $\nu_n>0$, we have $|\omega^{\nu_n}|\rightharpoonup^* \mu$ in $L^\infty([0,T];\mathcal M( \T^2))$, for some $\mu(t)\in \mathcal M( \T^2) $ which is non-atomic\footnote{A measure $\mu$ is said to be non-atomic if $\mu(\{x\})=0$ for all $x$.}  for a.e. $t$. Moreover, it holds 
\begin{equation}\label{jaemin please don't leave}
\lim_{R\to 0}\limsup_{\nu_n\to 0}\sup_{x_0\in\mathbb{T}^2}\int_{B_{R}(x_0)}|\omega^{\nu_n}(x,t)|\,dx = 0, \qquad \text{ for a.e. $t\in [0,T]$.}
\end{equation}
\end{enumerate}
\end{proposition}

\begin{proof}
We prove the two claims separately. 

\textsc{Proof of \ref{equi_LH}.} As it has already been done to prove \cref{estimate_lem}, to rigorously justify the next computations one should regularize the initial data, say $u^\nu_{0,\alpha}=u^\nu_0*\rho_\alpha$, and prove the corresponding claims for $\omega_\alpha^\nu=f^\nu_\alpha+\Omega^\nu_\alpha$. Since the latter will be uniform in $\alpha$, then the proof of \ref{equi_LH} will follow by letting $\alpha\rightarrow 0$. However, since this would unnecessarily burden the notation, we will skip the regularization step and only prove the claimed a priori estimates.

As usual, denote by $\{u^\nu\}_\nu$ the sequence of Leray--Hopf solutions with initial data $\{u^\nu_0\}_\nu$ and let $\omega_0^\nu=f_0^\nu+ \Omega^\nu_0$ be the splitting from the assumption \ref{H2}. We let $f^\nu$ and $\Omega^\nu$ to be the solutions to 
\begin{align}
\begin{cases}
\partial_tf^\nu + u^\nu\cdot\nabla f^\nu = \nu \Delta f^\nu\\
f^\nu (\cdot,0)=f^\nu_0,
\end{cases}
\qquad \text{ and } \qquad
\begin{cases}
\partial_t\Omega^\nu + u^\nu\cdot\nabla \Omega^\nu = \nu \Delta \Omega^\nu\\
\Omega^\nu(\cdot ,0)=\Omega^\nu_0.
\end{cases}
\end{align}
Since $\Omega_0^\nu\geq 0$, the maximum principle tells us $\Omega^\nu(t)\geq 0$ for all $t\ge0$. By uniqueness in the advection-diffusion equation with smooth velocity, it holds that  $\omega^\nu=f^\nu + \Omega^\nu$.
Moreover, the estimate \eqref{jaemin loves kimchi} is a direct consequence of \eqref{uniform_bounds_1} while the equi-integrability of $\{f^\nu\}_\nu$ comes from \ref{Ilostairpods} in \cref{equi_integrability} since we are assuming $\left\{f_0^\nu\right\}_{\nu}$ to be weakly compact in $L^1(\T^2)$.

\textsc{Proof of \ref{no_atom_b}.} By the uniform in viscosity bounds \eqref{sym1} and \eqref{jaemin loves kimchi} we can extract a subsequence $\nu_n\rightarrow 0$ such that 
$$
\omega^{\nu_n}\rightharpoonup^* \omega \qquad \text{ and } \qquad |f^{\nu_n}|\rightharpoonup^* f_+ \qquad \text{ in } L^\infty([0,T];\mathcal{M}(\T^2)).
$$
Moreover, the equi-integrability of $\{f^{\nu_n}\}_{n}$ implies that $f_+\in L^\infty([0,T];L^1(\T^2))$. Since $\omega^{\nu_n}=\curl u^{\nu_n}$ with $\{u^{\nu_n}\}_n$ bounded in $L^\infty([0,T];L^2(\T^2))$, it must hold $\omega \in L^\infty([0,T];H^{-1}(\T^2))$. In particular both $f_+(t)\in L^1(\T^2)$ and $\omega(t)\in \mathcal{M}(\T^2)\cap H^{-1}(\T^2)$ (see \cref{concentration_lem2} below) do not have atoms for a.e. $t\in [0,T]$. Thus, if $|\omega^{\nu_n}|\rightharpoonup^* \mu$ in $L^\infty([0,T];\mathcal{M}(\T^2))$, by 
\begin{equation}\label{smart bound}
|\omega^{\nu_n}|\leq |f^{\nu_n}| + \Omega^{\nu_n}\leq 2 |f^{\nu_n}| + \omega^{\nu_n},
\end{equation}
 we have $\mu\leq  \omega + 2f_+$ as elements in $L^\infty([0,T];\mathcal{M}(\T^2))$, which in turn implies that $\mu(t)\in \mathcal{M}(\T^2)$ does not have atoms for a.e. $t\in [0,T]$.

We are left to prove \eqref{jaemin please don't leave}. We start by noting that, since $u^{\nu_n}(t)\rightharpoonup u(t)$ in $L^2(\T^2)$ for a.e. $t\in [0,T]$ (see \cref{R:weak conv all times}), then $\omega^{\nu_n} (t)\rightharpoonup^*\omega(t)$ in $\mathcal{M}(\T^2)$ for a.e. $t\in [0,T]$. Thus the set 
$$
I:=\left\{ t\in [0,T] \,:\, \omega^{\nu_n} (t)\rightharpoonup^*\omega(t) \text{ in } \mathcal{M}(\T^2) ,\, \omega(t) \text{ does not have atoms} \right\}
$$
has full measure in $[0,T]$. Pick any $t\in I$, which from now on will be fixed. By the compactness of $\T^2$ we can find, possibly passing to a further subsequence\footnote{For the sake of clarity, let us specify that here it seems that the subsequence might depend on $t\in I$. However, the argument we provided proves that, starting from any subsequence $\nu_{n_k}$, we can find a further subsequence $\nu_{n_{k_j}}$ which enjoys \eqref{ugly bound}. This is enough to deduce that the limsup of the full sequence $\nu_n$ satisfies \eqref{ugly bound}.}, $x_{0,n}=x_{0,n}(R)$ such that $x_{0,n}\rightarrow z_0=z_0(R)\in \T^2$ and 
$$
\sup_{x_0\in \T^2} \int_{B_R(x_0)} |\omega^{\nu_n}(x,t)|\,dx=\int_{B_R(x_{0,n})}|\omega^{\nu_n}(x,t)|\,dx.
$$
If $n$ is chosen sufficiently large we can ensure that $B_R(x_{0,n})\subset B_{2R}(z_0)$.  Pick $\varphi_R$ such that 
$$
0\leq \varphi_R\leq 1, \quad \varphi_R \in C^0_c(B_{4R}(z_0)) \quad \text{and} \quad \varphi_R\big|_{B_{2R}(z_0)}\equiv 1.
$$
By \eqref{smart bound} we deduce
\begin{align}
    \limsup_{n\rightarrow \infty}\sup_{x_0\in \T^2} \int_{B_R(x_0)} |\omega^{\nu_n}(x,t)|\,dx&\leq  \limsup_{n\rightarrow \infty}\int_{B_{2R}(z_0)} |\omega^{\nu_n}(x,t)|\,dx\nonumber \\
    &\leq  \limsup_{n\rightarrow \infty}\int_{B_{4R}(z_0)} \varphi_R(x) |\omega^{\nu_n}(x,t)|\,dx\nonumber \\
    &\leq 2 \sup_n \int_{B_{4R}(z_0)} |f^{\nu_n}(x,t)|\,dx + \lim_{n\rightarrow \infty}\int_{B_{4R}(z_0)} \varphi_R(x)\omega^{\nu_n}(x,t)\,dx\nonumber \\
    &= 2 \sup_n \int_{B_{4R}(z_0)} |f^{\nu_n}(x,t)|\,dx +\int_{B_{4R}(z_0)} \varphi_R(x)\omega(x,t)\,dx
    \nonumber \\
    &\leq 2 \sup_n\sup_{x_0\in \T^2} \int_{B_{4R}(x_0)} |f^{\nu_n}(x,t)|\,dx + \sup_{x_0\in \T^2}|\omega_t |\left( B_{4R}(x_0)\right)\label{ugly bound},
\end{align}
where in the second last equality we have used that $t\in I$. Thanks to the (uniform in time) equi-integrability of $\{f^{\nu_n}(t)\}_n$ from \eqref{last equi int}, together with the fact that $\omega_t$ does not have atoms, \eqref{jaemin please don't leave} directly follows by letting $R\rightarrow 0$.
\end{proof}

\begin{lemma}\cite{delort1991existence}*{Lemma 1.2.5}\label{concentration_lem2}
Any  $\omega \in H^{-1}(\mathbb{T}^2) \cap \mathcal{M}(\mathbb{T}^2)$ does not have atoms, that is $\omega(\{x_0\})=\lim_{R\to 0}\omega(B_R(x_0))=0${, for any $x_0\in \mathbb{T}^2$}.
\end{lemma}

\subsection{Global weak solutions to Euler}

The goal of this section is to prove the following statement.
\begin{theorem}\label{T:Delort--Majda} Let $\{u^\nu_0\}_{\nu>0}$ satisfy \ref{H1} and \ref{H2} and $\{u^\nu\}_{\nu>0}$ be the corresponding sequence of Leray--Hopf solutions. Then, there exist $u_0\in L^2(\T^2)$, a subsequence $\nu_n>0$ and $u\in L^\infty([0,T];L^2(\mathbb{T}^2))$ such that $u^{\nu_{n}}\rightharpoonup^* u$ in $L^\infty([0,T];L^2(\mathbb{T}^2))$ and $u$ is an admissible weak solution to the Euler equation \eqref{E} with initial data $u_0$ in the sense of \cref{D: E admiss sol}. Moreover, for $\omega:=\curl u$, it holds
\begin{equation}\label{vort bounded in time}
    \|\omega(t)\|_{\mathcal M }\leq \sup_{\nu>0} \|\omega^\nu_0\|_{\mathcal M }\qquad \text{ for a.e. } t\in [0,T].
\end{equation}
\end{theorem}

The main argument of the proof aligns with the construction of weak solutions to the Euler equation in \cites{delort1991existence,vecchi19931}, with the key element of the proof lying in demonstrating that the sequence of approximate solutions does not generate a vorticity concentration in the limiting process. Although this argument has become standard, let us briefly describe it for the reader's convenience. Let us also emphasise that strictly speaking  Theorem~\ref{T:Delort--Majda} is not available in the literature, since the original proof by Delort \cite{delort1991existence} does not deal with the viscous  approximation while the Majda's one \cite{majda1993remarks} heavily relies on the positivity of the full vorticity, which is not consistent with being in the periodic setting.

By the energy inequality \eqref{Leray en ineq}, together with the $L^2$ compactness\footnote{Note that here (and thus in the assumptions of \cref{T:Delort--Majda}) and in Proposition~\ref{LH_sol_properties} the strong $L^2$ compactness of the initial data is not really necessary, and the  $L^2$ boundedness is enough. } of the initial velocities in \ref{H1}, it immediately follows that the sequence $\left\{ u^{\nu}\right\}_{\nu}$ enjoys a uniform bound
\[
\sup_{t>0,\,\nu>0}\| u^\nu(t)\|_{L^2} \leq \sup_{\nu>0}\| u^\nu_0\|_{L^2}<\infty.
\]
Thus, we infer that there exist $u_0\in L^2(\T^2)$ and  a subsequence $\nu_n$ and $u\in L^\infty([0,T]; L^2(\mathbb{T}^2)) $ such that $\nu_n\to 0$ and 
\begin{align}\label{convergence_NS_E}
u^{\nu_n} \rightharpoonup^* u \,  \text{ \,in }L^\infty([0,T]; L^2(\mathbb{T}^2)) \qquad \text{ and } \qquad u^{\nu_n}_0\to u_0\text{\, in $L^2(\mathbb{T}^2)$}.
\end{align}
\begin{remark}\label{R:weak conv all times}
A useful remark is that the convergence of $u^{\nu_n} \rightharpoonup^* u$ in $L^\infty([0,T]; L^2(\mathbb{T}^2))$ can be upgraded to $u^{\nu_n}(t)\rightharpoonup u(t)$ in $L^2(\T^2)$ for a.e. $t\in [0,T]$. Indeed, in view of the uniform bound of $\{u^{\nu_n}\}_n$ in $L^\infty([0,T]; L^2(\mathbb{T}^2))$, the Navier--Stokes equations \eqref{NS} automatically imply that $\{u^{\nu_n}\}_n$ stays bounded in $\Lip ([0,T]; H^{-N}(\mathbb{T}^2))$ for a sufficiently large $N\in \N$, from which the claim follows by Ascoli-Arzelà and the density of $C^\infty(\T^2)$ in $L^2(\T^2)$.
\end{remark}
The only remaining (and crucial!) step is to prove that $u$ is a weak solution to the Euler equation. Once this is accomplished, the fact that the solution is also admissible is an obvious consequence of \eqref{Leray en ineq} and  the lower semicontinuity of the $L^\infty([0,T];L^2(\T^2))$ norm, while \eqref{vort bounded in time} is a direct consequence of \eqref{sym1}.

Since $u^{\nu_n}$ is a strong solution to the Navier--Stokes system, it must satisfy the weak formulation \eqref{weak_formulation_NS}, that is, for any $\varphi\in C^2_c(\mathbb{T}^2\times [0,T))$ with $\div \varphi=0$, it holds
\begin{align}
\int_0^T \int_{\T^2} \left(u^{\nu_n} \cdot \partial_t\varphi + \nu_n u^{\nu_n} \cdot \Delta \varphi\right) \,dxdt +\int_{\T^2} u^{\nu_n}_0 \cdot \varphi (x,0)\,dx = - \int_0^T\int_{\mathbb{T}^2} u^{\nu_n}\otimes u^{\nu_n}  :\nabla \varphi  \,dxdt.
\end{align}
Taking $\nu_n\to 0$ and using the convergence properties in \eqref{convergence_NS_E}, the left-hand side converges to
\[
\int_0^T \int_{\T^2} u \cdot \partial_t\varphi  \,dxdt +\int_{\T^2} u_0 \cdot \varphi (x,0)\,dx.
\]
Hence in view of the weak formulation for the Euler equation, the proof of the theorem will be completed once we show, up to a further subsequence if necessary, that
\begin{align}\label{quadratic_convergence}
\lim_{\nu_n\to 0}\int_0^T\int_{\mathbb{T}^2} u^{\nu_n}\otimes u^{\nu_n}  :\nabla \varphi  \,dxdt = \int_0^T\int_{\mathbb{T}^2} u\otimes u  :\nabla \varphi  \,dxdt,
\end{align} 
for any $\varphi\in C^2_c(\mathbb{T}^2\times [0,T))$ with $\div \varphi=0$. 

As pointed out in \cite{delort1991existence},  a sufficient condition for \eqref{quadratic_convergence} to hold can be formulated in terms of the vorticity. Here below we state the  key lemma from the literature which leads to the needed concentration-compactness result.
\begin{lemma}\cite{scho95}*{Lemma 3.7}\label{concentration_lem}
Let $\omega^{\nu_n}:=\curl u^{\nu_n}$ with $\{u^{\nu_n}\}_n$ bounded in $L^\infty([0,T];L^2(\T^2))$. If $\{|\omega^{\nu_n}|\}_n$ is bounded in $L^\infty([0,T];\mathcal{M}(\mathbb{T}^2))$ and  \eqref{jaemin please don't leave} holds,  then \eqref{quadratic_convergence} holds.
\end{lemma}
In particular, \cref{concentration_lem} together with \ref{no_atom_b} from \cref{LH_sol_properties} concludes the proof of \cref{T:Delort--Majda}, which in turn gives \ref{T1} in \cref{T:main}.

\begin{remark}
The meticulous reader may notice that \cite{scho95}*{Lemma 3.7} assumes the slightly different condition 
$$
\lim_{N\rightarrow \infty}\lim_{R\to 0}\sup_{n\geq N}\sup_{x_0\in\mathbb{T}^2}\int_{B_{R}(x_0)}|\omega^{\nu_n}(x,t)|\,dx = 0, \qquad \text{ for a.e. $t\in [0,T]$.}
$$
However,  the very same proof goes through by assuming \eqref{jaemin please don't leave}.
\end{remark}

Strictly speaking, while slightly different statements of the above lemma can be found in several references (e.g. \cites{delort1991existence,vecchi19931,majda1993remarks,evans1994hardy}), they are all stated in the full space $\R^2$ but none of them deals with the case of the periodic box $\T^2$.  However, the proof of \cref{concentration_lem} can be easily localized on any bounded domain (e.g. \cite{delort1991existence}*{Section 2.3} and \cite{CLLV19}) from which the statement on the periodic box directly follows. Details are left to the reader, which may also consult \cite{Chem95}*{Chapter 6} and \cite{majda2002vorticity}*{Chapter 11}.

\section{No anomalous dissipation}\label{NAD}
In this section we prove that measure initial vorticities with positive singular part prohibit anomalous dissipation in the vanishing viscosity limit, in the two-dimensional setting. More precisely we prove the following result, which gives \ref{T2} and will conclude the proof of \cref{T:main}.

\begin{proposition}\label{P:no anom diss}
Let $\{u^\nu_0\}_{\nu>0}$ be a sequence of divergence-free vector fields satisfying \ref{H1} and \ref{H2}.   Let $\{u^\nu\}_{\nu>0}$ be the corresponding sequence of Leray--Hopf solutions to \eqref{NS} with $\nu\to 0$. Then 
\begin{equation}
    \limsup_{\nu\rightarrow 0} \nu\int_0^T\int_{\T^2} |\nabla u^\nu (x,t)|^2\,dxdt=0.
\end{equation}
\end{proposition}

\begin{proof} We argue by contradiction. Let us suppose to the contrary that there exists a subsequence $\nu_n$  such that $\nu_n\to 0$ as $n\to \infty$ and
\begin{align}\label{contrary}
\lim_{n\to \infty}\nu_n\int_0^T\int_{\T^2} \|\nabla u^{\nu_n} (t)\|^2_{L^2}\,dt=: \mathfrak{c}>0.
\end{align}
Let us decompose the dissipation as 
\begin{align}\label{Luigi_can_you_find_this}
\nu_n \int_0^{T} \| \nabla u^{\nu_n}(t)\|_{L^2}^2 \,dt = \nu_n \int_0^{\delta} \| \nabla u^{\nu_n}(t)\|_{L^2}^2 \,dt +  \nu_n \int_\delta^{T} \| \nabla u^{\nu_n}(t)\|_{L^2}^2 \,dt
\end{align}
for some $\delta>0$ which will be fixed later.

\textsc{Step 1:} No dissipation in $(0,\delta)$.

Here we deal with the first integral in \eqref{Luigi_can_you_find_this}.   By \cref{T:Delort--Majda} and \ref{H1} we can find a non-relabelled subsequence of $\left\{\nu_n\right\}_{n}$ such that  
\begin{align}\label{convergence_1}
u^{\nu_n}_0\to u_0 \text{ \, in $L^2(\mathbb{T}^2)$} \qquad \text{ and } \qquad  u^{\nu_n}\rightharpoonup^* u \text{ \, in }L^\infty([0,T];L^2(\T^2)),
\end{align} for some $u_0\in L^2(\mathbb{T}^2)$ and $u\in L^\infty([0,T];L^2(\T^2))$ which is an admissible weak solution to \eqref{E} with initial data $u_0$. In particular, see \cref{R:admiss are cont}, we can fix the precise representative $u\in C^0([0,T];L^2_{\rm w} (\T^2))$. 

To proceed, let $\eps$ be fixed so that
\begin{align}\label{small_suff}
0<\eps < \frac{\mathfrak{c}}{2}.
\end{align}By \cref{L:energy right cont} we can find $\delta>0$ such that 
\begin{equation}\label{en_cont_right}
 E_{u_0}- E_u (\delta)<\eps.
\end{equation}
Moreover, in view of \cref{R:weak conv all times}, by possibly reducing $\delta>0$ if necessary, we can also assume that $u^{\nu_n}(\delta)\rightharpoonup u(\delta)$ in $L^2(\T^2)$. From now on such $\delta$ will be fixed.
We split
\begin{align}\label{mathishard}
\nu_n \int_0^{\delta} \| \nabla u^{\nu_n}(t)\|_{L^2}^2 \,dt= E_{u_0^{\nu_n}}-E_{u^{\nu_n}}(\delta) =  \underbrace{E_{u_0^{\nu_n}} - E_{u_0}}_{ I_{n}} +E_{u_0} - E_{u}(\delta) + \underbrace{E_{u}(\delta)- E_{u^{\nu_n}}(\delta)}_{II_{n}(\delta)}.
\end{align}
By \eqref{convergence_1} we have $\limsup_{n\rightarrow \infty} I_n=0$. Moreover, by the lower semicontinuity of the $L^2$ norm under weak convergence we also get
$$
\limsup_{n\rightarrow \infty} II_n(\delta)=E_u(\delta)- \liminf_{n\rightarrow \infty}E_{u^{\nu_n}}(\delta) \leq 0.
$$
Thus, letting $n\rightarrow \infty$ in \eqref{mathishard} we achieve
\begin{equation}\label{no diss in zero delta}
\limsup_{n\rightarrow \infty} \nu_n \int_0^{\delta} \| \nabla u^{\nu_n}(t)\|_{L^2}^2 \,dt \leq E_{u_0} - E_u(\delta) <\eps < \frac{\mathfrak{c}}{2},
\end{equation}
where the last two inequalities are due to \eqref{en_cont_right} and \eqref{small_suff} respectively.

\textsc{Step 2:} No dissipation in $(\delta,T)$.    

Here we deal with the second integral in \eqref{Luigi_can_you_find_this}. Let us consider a sufficiently small number $\eta>0$, which will be chosen later, depending  on $\delta$ which was already fixed in \eqref{en_cont_right}.
Using the mollifier $\rho_{\eta\sqrt{\nu_n}}$ defined in \eqref{mollifier_1}, we can decompose 
\begin{equation}\label{split vortic molli}
\omega^{\nu_n}=\underbrace{\omega^{\nu_n}*\rho_{\eta\sqrt{\nu_n}}}_{=:\omega^{\nu_n}_1} + \underbrace{\left(\omega^{\nu_n} - \omega^{\nu_n}*\rho_{\eta\sqrt{\nu_n}}  \right)}_{=:\omega^{\nu_n}_2}.
\end{equation}
Again, from the definition of the mollifier in \eqref{mollifier_1}, we observe that
\begin{align}\label{lin_part1}
\| \omega^{\nu_n}_1(t)\|_{L^1}\le  \| \omega^{\nu_n}(t)\|_{L^1} \qquad \text{ and } \qquad \| \omega^{\nu_n}_1(t)\|_{L^\infty}\le \frac{1}{\eta^2\nu_n}\sup_{x_0\in \mathbb{T}^2}\int_{B_{\eta\sqrt{\nu_n}}(x_0)}|\omega^{\nu_n}(x,t)|\,dx.
\end{align}
In addition, a standard convergence rate of mollification gives us (e.g. \cite{majda2002vorticity}*{Lemma 3.5 (iv)})
\begin{align}\label{lin_part2}
\| \omega_2^{\nu_n}(t)\|_{L^2}^2\le C\eta^2\nu_n \| \nabla \omega^{\nu_n}(t)\|_{L^2}^2.
\end{align}
Now, let us consider the contribution to the dissipation of $\omega_1$ and $\omega_2$ separately,
\begin{align}\label{decomp_1_chap4}
\nu_n\int_{\delta}^{T}\| \omega^{\nu_n}(t)\|_{L^2}^2 \,dt \le C\left(\nu_n\int_{\delta}^{T}\| \omega^{\nu_n}_1(t)\|_{L^2}^2\, dt  + \nu_n\int_{\delta}^{T}\| \omega^{\nu_n}_2(t)\|_{L^2}^2 \,dt \right).
\end{align}

For the first term in the right-hand side, we use \eqref{lin_part1} to deduce
\begin{align}
\| \omega^{\nu_n}_1(t)\|_{L^2}^2&\le \| \omega^{\nu_n}_1(t)\|_{L^1}\| \omega^{\nu_n}_1(t)\|_{L^\infty}\le \frac{ \| \omega^{\nu_n}(t)\|_{L^1}}{\eta^2\nu_n}\sup_{x_0\in \mathbb{T}^2}\int_{B_{\eta\sqrt{\nu_n}}(x_0)}|\omega^{\nu_n}(x,t)|\,dx\\
& \le  \frac{ \sup_{\nu>0}\| \omega^\nu_0\|_{\mathcal{M}}}{\eta^2\nu_n}\sup_{x_0\in \mathbb{T}^2}\int_{B_{\eta\sqrt{\nu_n}}(x_0)}|\omega^{\nu_n}(x,t)|\,dx,
\end{align}
where the last inequality follows from \eqref{sym1}. Therefore, by  Fatou's lemma and \eqref{jaemin please don't leave} (up to possibly considering a further subsequence), we achieve
\[
\limsup_{n\to \infty}\nu_n\int_{\delta}^T \| \omega^{\nu_n}_1(t)\|_{L^2}^2\,dt \le  \frac{ \sup_{\nu>0}\| \omega^{\nu}_0\|_{\mathcal{M}}}{\eta^2}\int_{\delta}^T\limsup_{n\to \infty}\sup_{x_0\in\mathbb{T}^2}\int_{B_{\eta\sqrt{\nu_n}}(x_0)}|\omega^{\nu_n}(x,t)|\,dxdt =0.
\]

For the second term in \eqref{decomp_1_chap4}, we use \eqref{lin_part2} and \eqref{sym2} to derive
\[
\nu_n  \int_{\delta}^T \| \omega^{\nu_n}_2(t)\|_{L^2}^2\, dt \le C\eta^2 \nu_n^2 \int_{\delta}^T \| \nabla \omega^{\nu_n}(t)\|_{L^2}^2 \,dt\le C_{\delta,T}\eta^2 \sup_{\nu>0}\| \omega^\nu_0\|_{\mathcal{M}}^2.
\]
Therefore these estimates for each term in the right-hand side of \eqref{decomp_1_chap4}, combined with the usual singular integral operator estimate $\| \nabla u^\nu(t)\|_{L^2}=\| \omega(t)\|_{L^2}$, yield to
\[
\limsup_{n\to \infty}\nu_n\int_{\delta}^{T}\| \nabla u^{\nu_n}(t)\|_{L^2}^2 \,dt=\limsup_{n\to \infty}\nu_n\int_{\delta}^{T}\| \omega^{\nu_n}(t)\|_{L^2}^2\, dt \le  C_{\delta,T}\eta^2 \sup_{\nu>0}\| \omega^\nu_0\|_{\mathcal{M}}^2.
\]
Now, since $\delta>0$ is already fixed, we can choose $\eta$ sufficiently small depending on $C_{\delta,T}$ and $ \sup_{\nu>0}\| \omega^\nu_0\|_{\mathcal{M}}$ in the above inequality so that
\[
\limsup_{n\to \infty}\nu_n\int_{\delta}^{T}\| \nabla u^{\nu_n}(t)\|_{L^2}^2 \,dt< \frac{\mathfrak{c}}{2}.
\]
Finally, we combine this with \eqref{no diss in zero delta} and conclude that taking $n\to \infty$ in \eqref{Luigi_can_you_find_this} leads us to
\[
\limsup_{n\to\infty} \nu_n \int_0^{T} \| \nabla u^{\nu_n}(t)\|_{L^2}^2 \,dt < \mathfrak{c},
\]
which contradicts \eqref{contrary}.
\end{proof}

\section{Further discussions and remarks}\label{S:further remarks}
This section aims to provide a technically stronger version of our main \cref{T:main}, which in particular leads to a natural and quite geometrical \quotes{singularity statement} interpretation when anomalous dissipation happens. We also wish to give some additional comments on our proof which should help in clarifying on the choices we have made.

\subsection{A more general statement \& singularity interpretation}
As it is apparent from the proof of \cref{P:no anom diss}, the key elements which prohibit anomalous dissipation are the strong $L^2(\T^2)$ compactness of the sequence of initial velocities $\{u_0^\nu\}_\nu$ and the absence of atomic concentration in the vorticity measure $\{|\omega^\nu|\}_\nu$. Assuming $\{u_0^\nu\}_\nu\subset L^2(\T^2)$ strongly compact is natural since otherwise wild oscillations in the initial data allow for anomalous dissipation even in the (linear) heat equation. Thus, the only non-trivial requirement is that on the vorticity. Our assumption \ref{H2} on the initial vorticities is just a particular case which implies that  $\{|\omega^\nu|\}_\nu$ does not have atomic concentrations in space, for almost every time. In view of that, let us state a more general version of \cref{T:main} which highlights what is really needed to run the argument.
\begin{theorem}\label{T:main more general}
Let $\{u^\nu_0\}_{\nu>0}\subset L^2(\T^2)$ be a sequence of divergence-free vector fields satisfying \ref{H1} and   let $\{u^\nu\}_{\nu>0}$ be the corresponding sequence of Leray--Hopf solutions to \eqref{NS} with $\nu\to 0$. Assume that $\{ \omega^\nu\}_{\nu>0}$ is bounded in $L^\infty([0,T];L^1(\T^2))$. If
\begin{equation}
    \label{space_time atomic concentration}
    \lim_{R\rightarrow 0} \limsup_{\nu\rightarrow 0} \int_0^T \sup_{x_0\in \T^2} \int_{B_R(x_0)} |\omega^\nu(x,t)|\,dxdt=0,
\end{equation}
then
\begin{equation}
    \limsup_{\nu\rightarrow 0} \nu\int_0^T\int_{\T^2} |\nabla u^\nu (x,t)|^2\,dxdt=0.
\end{equation}
\end{theorem}
The proof of the previous theorem follows by the very same argument of that of \cref{P:no anom diss} without any modification. In \cref{LH_sol_properties} we have proved that the assumption \ref{H2}  guarantees the validity of \eqref{space_time atomic concentration}. 
Having established the more general \cref{T:main more general}, the corresponding singularity-type statement for vanishing viscosity sequences readily follows. 

\begin{corollary}\label{C:singularity statement}
Let $\{u^\nu_0\}_{\nu>0}\subset L^2(\T^2)$ be a sequence of divergence-free vector fields satisfying \ref{H1} and   let $\{u^\nu\}_{\nu>0}$ be the corresponding sequence of Leray--Hopf solutions to \eqref{NS} with $\nu\to 0$. Assume that $\{ \omega^\nu\}_{\nu>0}$ is bounded in $L^\infty([0,T];L^1(\T^2))$. If 
\begin{equation}
    \liminf_{\nu\rightarrow 0} \nu\int_0^T\int_{\T^2} |\nabla u^\nu (x,t)|^2\,dxdt>0,
\end{equation}
then, for any subsequence $\{\omega^{\nu_n}\}_n$, it must hold
\begin{equation}
    \label{space_time atomic concentration 2}
    \lim_{R\rightarrow 0} \limsup_{n \rightarrow \infty} \int_0^T \sup_{x_0\in \T^2} \int_{B_R(x_0)} |\omega^{\nu_n}(x,t)|\,dxdt >0.
\end{equation}
\end{corollary}
Let us emphasize that here we are neither claiming that, in general, anomalous dissipation in two dimensions can occur, nor that the concentration of the vorticity measure can provide an effective mechanism for it. However, \cref{C:singularity statement} highlights that \eqref{space_time atomic concentration 2} is at least a necessary condition for anomalous dissipation to happen, which we hope might give some new/different insights on the topic. Also, as opposite as the usual singularity-type statements on the blow-up of every Onsager's subcritical norm, our condition \eqref{space_time atomic concentration 2} is perhaps more geometric than analytic: for anomalous dissipation to happen, regions with intense positive and negative vorticity must collapse and give positive mass to points.

\subsection{Further comments on the proof of \cref{T:main}}
Here we wish to comment a bit more on the proof of \cref{P:no anom diss}, more specifically on why the vorticity measure $|\omega^\nu|$ appears on a disk of size proportional $\sqrt{\nu}$.

It is clear that as soon as in the splitting \eqref{split vortic molli} we mollify at some length scale $\alpha$, that is
\begin{equation}
\omega^{\nu}=\underbrace{\omega^{\nu}*\rho_{\alpha}}_{=:\omega^{\nu}_1} + \underbrace{\left(\omega^{\nu} - \omega^{\nu}*\rho_{\alpha}  \right)}_{=:\omega^{\nu}_2},
\end{equation}
then by the very same computations we get
\begin{align}\label{remarks balance 1}
\nu \int_\delta^T \| \omega^{\nu}_1(t)\|_{L^2}^2 \,dt\le   C \frac{\nu }{\alpha^2}\int_\delta^T \left( \sup_{x_0\in \mathbb{T}^2}\int_{B_{\alpha}(x_0)}|\omega^{\nu}(x,t)|\,dx\right) dt
\end{align}
and 
\begin{align}\label{remarks on the balance 2}
\nu \int_\delta^T \| \omega^{\nu}_2(t)\|_{L^2}^2 \,dt\le   C_{\delta, T} \frac{\alpha^2}{\nu}.
\end{align}
In particular, the choice $\alpha=\sqrt{\nu}$ is the only one which makes both terms bounded in the viscosity parameter. With this choice, the right hand side in \eqref{remarks balance 1} will vanish when $\nu\rightarrow 0$ since the vorticity measure does not concentrate.  The additional small parameter $\eta$  in \eqref{split vortic molli} is then used to guarantee that also the right hand side in \eqref{remarks on the balance 2} can be made arbitrarily small.

However, besides optimizing the estimate, the appearance of the vorticity measure on a disk of size $\sqrt{\nu}$ has a more intrinsic reason. Indeed, a direct application of \cref{P:vort_est_2d} and the Cauchy--Schwartz inequality yields
\begin{equation}
    \label{goal remark}
    \int_{B_{\sqrt{\nu}}(x_0)}|\omega^{\nu}(x,t)|\,dx\leq \frac{\|u_0^\nu\|_{L^2}}{\sqrt{2 t}}\qquad \text{for all } t>0.
\end{equation}
In particular, for any $t>0$, the mass of $|\omega^\nu(t)|$ on disks of radius proportional to $\sqrt \nu$ stays always bounded, in $\nu\rightarrow 0$, for general Leray--Hopf solutions without any  assumption on the initial vorticity.
Thus, as soon as the sequence of initial data $\{u_0^\nu\}_\nu$ stays bounded $L^2(\T^2)$, the corresponding sequence of Leray--Hopf weak solutions enjoys
$$
\lim_{\nu\rightarrow 0} \frac{\alpha(\nu)}{\sqrt{\nu}}=0 \quad \Longrightarrow \quad \lim_{\nu\rightarrow 0} \sup_{x_0 \in \T^2} \int_{B_{\alpha(\nu)}(x_0)}|\omega^{\nu}(x,t)|\,dx=0 \qquad \text{for all } t>0.
$$
This should be coherent with the prediction (via dimensional analysis) that below the Kolmogorov length scale, which we recall in two dimensions to be proportional to $\sqrt{\nu}$, the dynamics is viscosity dominated and thus all quantities behave as if they were smooth.
\color{black}

\section{Energy conservation vs strong compactness}\label{S:en cons and compactness}
We have established that no anomalous dissipation can occur if the initial vorticities satisfy \ref{H2}. A simple argument shows that having both no anomalous dissipation and kinetic energy conservation of the limit is equivalent to the $L^2(\T^2\times [0,T])$ strong compactness of the sequence $\{u^\nu\}_\nu$. Since the latter property does not seem to be expected in this context,  one should not believe that our result can be upgraded to kinetic energy conservation of the limit. However,  since this (very delicate!) issue is quite unclear to the authors, we prefer to avoid any claim.

Although the relation between strong compactness of $\{u^\nu\}_\nu$ and energy conservation of the limit $u$ was already known in more generality (see \cite{LMP21}*{Theorem 2.11}), here we give a very elementary proof when no anomalous dissipation is assumed. {We emphasise that the general statement \cite{LMP21}*{Theorem 2.11} proves the quite interesting feature that the mathematical analysis of two-dimensional turbulence is very different than that in three dimensions. Indeed, in the latter there is a great consensus on the coexistence of $L^2$ strong compactness of the velocity and anomalous dissipation.}

\begin{proposition}\label{P:en cons and compact}
Let $\{u_0^\nu\}_{\nu>0}\subset L^2(\T^2)$ be a sequence of divergence-free initial data such that $u_0^\nu \rightarrow u_0$ in $L^2(\T^2)$ and denote by $u^\nu$ the corresponding sequence of Leray--Hopf solutions to \eqref{NS} with $\nu\rightarrow 0$. Assume that 
$$
\limsup_{\nu\rightarrow 0} \nu\int_0^T  \int_{\T^2} \left|\nabla u^{\nu}(x,t)\right|^2\,dxdt =0
$$
and $u^\nu \rightharpoonup u$ in $L^2(\T^2\times [0,T])$. Then, denoting by $E_u$ the kinetic energy of the limit $u$, we have that $E_u(t)=E_{u_0}$ for a.e. $t\in [0,T]$ implies $u^\nu \rightarrow u$ in $L^p([0,T];L^2(\T^2))$ for all $p\in [1,\infty)$. Conversely, if $u^\nu \rightarrow u$ in $L^1([0,T];L^2(\T^2))$ then $E_u(t)=E_{u_0}$ for a.e. $t\in [0,T]$.
\end{proposition}
\begin{proof}
Assume that $E_u(t)=E_{u_0}$ for a.e. $t\in [0,T]$. Since by 
\begin{equation}\label{en balance}
E_{u^\nu} (t)+ \nu \int_0^t \int_{\T^2} |\nabla u^\nu (x,s)|^2 \,dxds= E_{u^\nu_0} 
\end{equation}
we have that $\{u^\nu\}_\nu$ is bounded in $L^\infty([0,T];L^2(\T^2))$, it is enough to prove that $u^\nu \rightarrow u$ in $L^2(\T^2\times [0,T])$. Integrating \eqref{en balance} in time and sending $\nu \rightarrow 0$ yields to 
$$
\int_0^T E_u(t)\,dt\leq \liminf_{\nu \rightarrow 0} \int_0^T E_{u^\nu}(t)\,dt\leq \limsup_{\nu \rightarrow 0} \int_0^T E_{u^\nu}(t)\,dt =T E_{u_0} =\int_0^T E_u(t)\,dt.
$$
In particular 
$$
\|u^\nu\|_{L^2(\T^2\times [0,T])}\rightarrow \|u\|_{L^2(\T^2\times [0,T])},
$$
which together with $u^\nu \rightharpoonup u$ in $L^2(\T^2\times [0,T])$ implies strong $L^2$ convergence.

Assume now $u^\nu \rightarrow u$ in $L^1([0,T];L^2(\T^2))$. Then, up to possibly extract a (non-relabelled) subsequence, we have $u^\nu (t)\rightarrow u(t)$ in $L^2(\T^2)$ for a.e. $t\in [0,T]$. Thus, by letting $\nu \rightarrow 0$ in \eqref{en balance}, and since $u_0^\nu \rightarrow u_0$ in $L^2(\T^2)$, we conclude $E_u(t)=E_{u_0}$ for a.e. $t\in [0,T]$. 
\end{proof}

\bibliographystyle{plain} 
\bibliography{biblio}

\end{document}